\documentclass[a4paper,12pt]{article}
\usepackage{amsmath}
\usepackage{amssymb}
\usepackage{amscd}
\usepackage{amsthm}
\usepackage{graphicx}
\newtheorem{theorem}{Theorem}
\newtheorem{lemma}[theorem]{Lemma}

\newtheorem*{mtheorem}{Main Theorem}
\newtheorem{corollary}[theorem]{Corollary}

\usepackage{pdfsync}
\def \PG{\mathrm{PG}}

\def\V{\mathrm{V}}

\def\B{\mathcal{B}}
\def\D{\mathcal{D}}

\def\F{\mathbb{F}}

\def\L{\mathcal{L}}

\title{A small minimal blocking set in $\PG(n,p^t)$, spanning a $(t-1)$-space, is linear}
\author{Peter Sziklai \and  Geertrui Van de Voorde}
\begin{document}

\maketitle

\begin{abstract} 
In this paper, we show that a small minimal blocking set with exponent $e$ in $\PG(n,p^t)$, $p$ prime, spanning a $(t/e-1)$-dimensional space, is an $\F_{p^e}$-linear set, provided that $ p>5(t/e)-11$. As a corollary, we get that all small minimal blocking sets in $\PG(n,p^t)$, $p$ prime, $p>5t-11$, spanning a $(t-1)$-dimensional space, are $\F_p$-linear, hence confirming the linearity conjecture for blocking sets in this particular case.

\end{abstract}

\section{Introduction}
In this section, we introduce the necessary background and notation. If $V$ is a vector space, then we denote the corresponding projective space by $\PG(V)$. If $V$ has dimension $n+1$ over the finite field $\F_q$, with $q$ elements, $q=p^t$, $p$ prime, then we also write $V$ as $\V(n+1,q)$ and $\PG(V)$ as $\PG(n,q)$. 

A {\em blocking set} in $\PG(n,q)$ is a set $B$ of points such that every hyperplane of $\PG(n,q)$ contains at least one point of $B$. Such a blocking set is sometimes called a {\em $1$-blocking set}, or a {\em blocking set with respect to hyperplanes}. A blocking set $B$ is called {\em small} if $|B|<3(q+1)/2$ and {\em minimal} if no proper subset of $B$ is a blocking set. 

A point set $S$ in $\PG(V)$, where $V=\V(n+1,p^t)$, is called  $\F_{q_0}${\em-linear} if there exists a subset $U$ of $V$ that forms an $\F_{q_0}$-vector space for some $\F_{q_0} \subset {\mathbb{F}}_{p^t}$, such that $S=\B(U)$, where 
$$\B(U):=\{\langle u \rangle_{\mathbb{F}_{p^t}}~:~u \in U\setminus \{0\}\}.$$

We have a one-to-one correspondence between the points of $\PG(n,q_0^h)$ and the elements of a Desarguesian $(h-1)$-spread $\D$ of $\PG(h(n+1)-1,q_0)$. This gives us a different view on linear sets; namely, an $\F_{q_0}$-linear set is a set $S$ of points of $\PG(n,q_0^h)$ for which there exists a subspace $\pi$ in $\PG(h(n+1)-1,q_0)$ such that the points of $S$ correspond to the elements of $\D$ that have a non-empty intersection with $\pi$. We identify the elements of $\D$ with the points of $\PG(n,q_0^h)$, so we can view $\B(\pi)$ as a subset of $\D$, i.e.
$$\B(\pi)=\{R\in \D|R\cap \pi\neq \emptyset\}.$$

For more information on this approach to linear sets, we refer to \cite{linearsets}. 

The {\em linearity conjecture} for blocking sets (see \cite{sziklai}) states that
\begin{itemize}
\item[(LC)] All small minimal blocking sets in $\PG(n,q)$ are linear sets.
\end{itemize}

Up to our knowledge, this is the complete list of cases in which the linearity conjecture for blocking sets in $\PG(n,p^t)$, $p$ prime, with respect to hyperplanes, has been proven.

\begin{itemize}
\item $t=1$ (for $n=2$, see \cite{blok}; for $n>2$, see \cite{heim})
\item $t=2$ (for $n=2$, see \cite{TS:97}; for $n>2$, see \cite{Storme-Weiner})
\item $t=3$ (for $n=2$, see \cite{pol}; for $n>2$, see \cite{Storme-Weiner})
\item $B$ is of R\'edei-type, i.e., there is a hyperplane meeting $B$ in $\vert B \vert-p^t$ points  (for $n=2$, see \cite{simeon,redei2}; for $n>2$, see \cite{redei})
\item dim$\langle B \rangle=t$ (see \cite{sz}).
\end{itemize}

In this paper, we show that if dim$\langle B \rangle=t-1$,  and the characteristic of the field is sufficiently large, $B$ is a linear set, as a corollary of the main theorem.

\begin{mtheorem}A small minimal blocking set $B$ in $\PG(n,q)$, with exponent $e$, $q=p^t$, $p$ prime, $q_0:=p^e$, $q_0\geq 7$, $t/e=h$, spanning an $(h-1)$-dimensional space is an $\F_{q_0}$-linear set.
\end{mtheorem}

\section{The intersection of a small minimal blocking set and a subspace}
A subspace clearly meets an $\F_p$-linear set in $1$ mod $p$ or $0$ points. The following theorem shows that for a small minimal blocking set, the same holds.
\begin{theorem}\label{1modp}
{\rm \cite[Theorem 2.7]{sz}} If $B$ is a small minimal blocking set in $\PG(n,p^t)$, $p$ prime, then $B$ intersects every subspace of $\PG(n,p^t)$ in $1$ mod $p$ or $0$ points.
\end{theorem}
From this theorem, we get that every small minimal blocking set $B$ in $\PG(n,p^t)$, $p$ prime, has an {\em exponent e} $\geq 1$, which is the largest integer for which every hyperplane intersects $B$ in $1$ mod $p^e$ points.

\subsection{The intersection with a line}
The following theorem by Sziklai characterises the intersection of particular lines with a small minimal blocking set as a linear set. \begin{theorem}
{\rm \cite[Corollary 5.2]{sziklai}} \label{sublines}Let $B$ be a small minimal blocking set with exponent $e$ in $\PG(n,q)$, $q=p^t$, $p$ prime. If for a certain line $L$, $\vert L\cap B \vert=p^e+1$, then $\F_{p^e}$ is a subfield of $\F_q$ and $L\cap B$ is a subline $\PG(1,p^e)$.
\end{theorem}

Using the $1$ mod $p$-result (Theorem \ref{1modp}), it is not too hard to derive an upper bound on the size of a small minimal blocking set in $\PG(n,q)$ as done in \cite{reduction}. This bound is a weaker version of the bound in Corollary 5.2 of \cite{sziklai}.
\begin{lemma}{\rm \cite[Lemma 1]{reduction}} \label{size}
The size of a small minimal blocking set $B$ with exponent $e$ in $\PG(n,q_0^h)$, $q_0:=p^e\geq 7$, $p$ prime, is at most $q_0^{h}+q_0^{h-1}+q_0^{h-2}+3q_0^{h-3}$.
\end{lemma}

In this paper, we will make use of the fact that we can find lower bounds on the number of secant lines to a small minimal blocking set. In the next lemma, one considers the number of $(q_0+1)$-secants to the blocking set $B$, which will give a linear intersection with the blocking set by Theorem \ref{sublines}.
\begin{lemma}{\rm \cite[Lemma 4]{reduction}} \label{lemma1} A point of a small minimal blocking set $B$ with exponent $e$ in $\PG(n,q_0^h)$, $q_0:=p^e\geq 7$, $p$ prime, lying on a $(q_0+1)$-secant, lies on at least $q_0^{h-1}-4q_0^{h-2}+1$ $(q_0+1)$-secants.
\end{lemma}
For the proof of Lemma \ref{lemma3}, we will make use of the concept of point exponents of a blocking set and the well-known fact that the projection of a small minimal blocking set is a small minimal blocking set.

\begin{lemma}{\rm\cite[Corollary 3.2]{sz}} \label{projection} Let $n\geq 3$. The projection of a small minimal blocking set in $\PG(n,q)$, from a point $Q\notin B$ onto a hyperplane skew to $Q$, is a small minimal blocking set in $\PG(n-1,q)$.
\end{lemma}
The exponent $e_P$ of a point $P$ of a small minimal blocking set $B$ is the largest number for which every line through $P$ meets $B$ in $1$ mod $p^{e_P}$ or $0$ points.
The following lemma is essentially due to Blokhuis.

\begin{lemma} {\rm (See \cite[Lemma 2.4(1)]{aart})} \label{numbersecants}If $B$ is a small minimal blocking set in $\PG(2,q)$, $q=p^t$, $p$ prime, with $|B|=q+\kappa$, and $P$ is a point with exponent $e_P$, then the number of secants to $B$ through $P$, is at least $$(q-\kappa+1)/p^{e_P}+1.$$
\end{lemma}

\begin{lemma} \label{lemma3} A point $P$ with exponent $e_P=2e$ of a small minimal blocking set $B$ with exponent $e$ in $\PG(n,q_0^h)$, $q_0:=p^e\geq 7$, $p$ prime, lies on at least $q_0^{h-2}-q_0^{h-3}-q_0^{h-4}-3q_0^{h-5}+1$ secant lines to $B$.
\end{lemma}
\begin{proof} If $n=2$, Lemma \ref{size}, together with Lemma \ref{numbersecants}, shows that the number of secant lines to $B$ is at least $(q_0^h-q_0^{h-1}-q_0^{h-2}-3q_0^{h-3}+1)/q_0^2+1\geq q_0^{h-2}-q_0^{h-3}-q_0^{h-4}-3q_0^{h-5}+1$.

If $n>2$, then let $L$ be a line through $P$, meeting $B$ in $q_0^2+1$ points. There exists a point $Q$, not on $B$, lying only on tangent lines to $B$. 
Let $\tilde{B}$ be the projection of $B$ from $Q$ onto a hyperplane through $L$. By Lemma \ref{projection}, $\tilde{B}$ is a small minimal blocking set in $\PG(n-1,q)$. It is clear that every line through $P$ meets $\tilde{B}$ in $1$ mod $q_0^2$ or $0$ points, and that there is a line, namely $L$, meeting $\tilde{B}$ in $1+q_0^2$ points, so $e_{P}=2e$ in the blocking set $\tilde{B}$. It follows that the number of secant lines through a point $P$ with exponent $2e$ to $B$ is at least the number of secant lines through the point $P$ with exponent $2e$ to $\tilde{B}$ in $\PG(n-1,q_0^h)$. Continuing this process, we see that this number is at least the number of secant lines through the point $P$ with exponent $2e$ in a small minimal blocking set $\tilde{B}$ in $\PG(2,q_0^h)$, and the statement follows.
\end{proof}

\subsection{The intersection with a plane}

In the following lemma, we will distinguish planes according to their intersection size with a small minimal blocking set. We will call a plane with $q_0^2+q_0+1$ non-collinear points of $B$ a {\em good} plane, while all other planes will be called {\em bad}. Note that also planes meeting $B$ in only points on a line, or skew to $B$ are called {\em bad}. The following lemma shows that good planes meet a small minimal blocking set in a linear set.

\begin{lemma} \label{eerste}If $\Pi$ is a plane of $\PG(n,q)$ containing at least $3$ non-collinear points of a small minimal  blocking set $B$ in $\PG(n,q)$, with exponent $e$, $q=p^t$, $p$ prime, $q_0:=p^e$, then
\begin{itemize}
\item[(i)]$q_0^2+q_0+1\leq \vert B\cap \Pi\vert.$
\item[(ii)] If $\vert B \cap \Pi\vert=q_0^2+q_0+1$, then $B \cap \Pi$ is $\F_{q_0}$-linear. 
\item[(iii)] If $\vert B \cap \Pi\vert >q_0^2+q_0+1$, then $\vert B \cap \Pi \vert \geq 2q_0^2+q_0+1$.
\end{itemize}
\end{lemma}

\begin{proof}
(i) By Lemma \ref{1modp}, every line meets $B$ in $1$ mod $q_0$ or 0 points. Since we find $3$ non-collinear points, it is easy to see that $\vert B\cap \Pi\vert \geq q_0^2+q_0+1$.

(ii) From the previous argument, we easily see that if $\vert B\cap \Pi\vert = q_0^2+q_0+1$, then every line in $\Pi$ contains $0$, $1$ or $q_0+1$ points of $B$. Suppose that there exist two $(q_0+1)$-secants that meet in a point, not in $B$, then the number of points in $\Pi \cap B$ is at least $q_0^2+q_0+1+q_0$. Hence, every two $(q_0+1)$-secants meet in a point of $B$. Moreover, through two points of $B\cap \Pi$, there is a unique $(q_0+1)$-secant, so $B$ meets $\Pi$ in an $\F_{q_0}$-subplane.

(iii) By Theorem \ref{1modp}, if there is a line $L$ of $\Pi$ containing more than $(q_0+1)$ points of $B$, then $\vert L \cap B\vert \geq 2q_0+1$, and $\vert \Pi\cap B\vert \geq 2q_0^2+q_0+1$. So from now on, we may assume that every line meets $B$ in $0$, $1$ or $q_0+1$ points. If there is an $\F_{q_0}$-subplane strictly contained in $\Pi \cap B$, then clearly $\vert B \cap \Pi\vert \geq q_0^3+q_0^2+1$, so we may assume that there is no $\F_{q_0}$-subplane contained in $\Pi\cap B$.

Let $L$ be a $(q_0+1)$-secant in $\Pi$, let $P$ be a point of $B\cap L$, let $Q$ be a point of $B\setminus L$ and let $M$ be the line $PQ$. From Theorem \ref{sublines}, we know that  $L\cap B$ and $M\cap B$ are sublines over $\F_{q_0}$. 
These sublines define a unique $\F_{q_0}$-subplane $\Pi_0$. 
Let $N_1$ be a line, not through $P$, through a point of $L\cap B$, say $R_1$ and of $M\cap B$, say $R_2$. Let $N_2$ be another line, not through $R_1$ or $R_2$, meeting $L$ in a point $R_3$ of $B$ and $M$ in a point $R_4$ of $B$. If $T$ is the intersection point of $N_1$ and $N_2$, then $T$ belongs to the subplane $\Pi_0$. 

Now suppose that $T$ is a point of $B$, then $N_1$ meets $B$ in a subline, containing $3$ points of the subline $\Pi_0\cap N_1$. This implies that the subline $N_1\cap B$ is completely contained in $B$. The same holds for the subline $N_2\cap B$, and repeating the same argument, for every subline through $T$ meeting $L$ and $M$ in points of $B$, different from $P$. Again repeating the same argument, for a point $T'\neq T$ on $N_1$, not on $L$ or $M$, yields that $\Pi_0$ is contained in $B$, a contradiction. This implies that the $q_0-1$ points of $B$ on the line $N_1$, not on $L$ or $M$, are different from the $q_0-1$ points of $B$ on the line $N_2$, not on $L$ or $M$. Varying $N_1$ and $N_2$ over all lines meeting $L$ and $M$ in points of $B$, we get that there are at least $q_0^2(q_0-1)+2q_0+1$ points in $B\cap \Pi$.
\end{proof}

To avoid abundant notation, we continue with the following hypothesis on $B$.
\begin{itemize}
\item[] $B$ is a small minimal blocking set in $\PG(n,q)$, with exponent $e$, $q=p^t$, $p$ prime, $q_0:=p^e$, $t/e=h$, {\bf spanning an $(h-1)$-dimensional space}.
\end{itemize}

\begin{lemma}\label{lemma2}
A plane of $\PG(n,q)$ contains at most $q_0^3+q_0^2+q_0+1$ points of $B$.
\end{lemma}
\begin{proof}
Suppose there exists a plane $\Pi$ with more than $q_0^3+q_0^2+q_0+1$ points of $B$, then, by Theorem \ref{1modp}, $\vert \Pi \cap B\vert \geq q_0^3+q_0^2+2q_0+1$. We prove by induction that, for all $2\leq k\leq h-1$, there is a $k$-space, containing at least $(q_0^{k+2}-1)/(q_0-1)+q_0^{k-1}$ points of $B$. The case $k=2$ is already settled, so suppose there is a $j$-space $\Pi_j$, $j<h-1$, containing at least $(q_0^{j+2}-1)/(q_0-1)+q_0^{j-1}$ points of $B$. Since $B$ spans an $(h-1)$-space and $j<h-1$, there is a point $Q$ in $B$, not in $\Pi_j$. Because a line containing two points of $B$ contains at least $q_0+1$ points of $B$, this implies that $\vert \langle Q,\Pi_j\rangle\cap B\vert\geq (q_0^{j+3}-1)/(q_0-1)+q_0^{j}$. By induction, we obtain that $B$ contains at least $(q_0^{h+1}-1)/(q_0-1)+q_0^{h-2}$ points, a contradiction, since $\vert B \vert \leq q_0^{h}+q_0^{h-1}+q_0^{h-2}+3q_0^{h-3}$.
\end{proof}

\begin{lemma} \label{goodplanes}
Let $L$ be a $(q_0+1)$-secant to $B$. Then either $L$ lies on at least $q_0^{h-2}-4q_0^{h-3}+1$ good planes, or $L$ lies on bad planes only. In the latter case, all planes with points of $B$ contain at least $q_0^3+q_0+1$ points of $B$ outside of $L$.
\end{lemma}
\begin{proof}
Let $Q$ be a point on $L$, not on $B$. We project $B$ from $Q$ onto a hyperplane $H$, not through $Q$, and denote the image of this projection by $\tilde{B}$. Let $P$ be the point $L\cap H$. It follows from Lemma \ref{projection}, that $\tilde{B}$ is a small minimal blocking set. Since every subspace meets $B$ in $1$ mod $q_0$ or $0$ points, every subspace meets $\tilde{B}$ in $1$ mod $q_0$ or $0$ points. Suppose that $P$ has exponent $e_P=1$, then it follows from Lemma \ref{lemma1} that $P$ lies on at least $q_0^{h-1}-4q_0^{h-2}+1$ $(q_0+1)$-secants. This means that there are at least $q_0^{h-1}-4q_0^{h-2}+1$ planes through $L$ containing at least $q_0^2+q_0+1$ points of $B$, which implies that $\vert B\vert \geq q_0^2(q_0^{h-1}-4q_0^{h-2}+1)$, a contradiction since $\vert B \vert \leq q_0^h+q_0^{h-1}+q_0^{h-2}+3q_0^{h-3}$ by Lemma \ref{size}.

If $P$ has exponent $e_P$ at least $4$, we get that the planes through $L$ which contain a point of $B$, not on $L$, contain at least $q_0^4+q_0+1$ points of $B$, which is impossible by Lemma \ref{lemma2}. We conclude that $P$ has exponent $e_P=2$ or $e_P=3$. If $P$ has exponent $e_P=3$, then every plane through $L$ that contains a point of $B$ not on $L$, contains at least $q_0^3+q_0+1$ points, and hence, all planes through $L$ are bad.

Finally, if $P$ has exponent $2$, we know from Lemma \ref{lemma3} that there are at least $s=q_0^{h-2}-q_0^{h-3}-q_0^{h-4}-3q_0^{h-5}+1$ secant lines through $P$, which implies that there are at least $s$ planes through $L$ containing a point of $B$ outside of $L$. Suppose $t$ of the $s$ planes are bad, then, using Lemma \ref{eerste}(iii), $B$ contains at least
$t(2q_0^2)+(s-t)(q_0^2)+q_0+1$ points. If we put $t=3q_0^{h-3}-q_0^{h-4}-3q_0^{h-5}+1$, we get a contradiction since $\vert B \vert \leq q_0^h+q_0^{h-1}+q_0^{h-2}+3q_0^{h-3}$ by Lemma \ref{size}.
\end{proof}

\begin{lemma}\label{eenslecht}
A point $P$ of $B$ lying on a $(q_0+1)$-secant, lies on at most one $(q_0+1)$-secant $L$ that lies on only bad planes.
\end{lemma}

\begin{proof}
Let $P$ be a point of $B$ lying on a $(q_0+1)$-secant and let $L$ be a line through $P$ that only lies on bad planes. From Lemma \ref{lemma2} and Lemma \ref{goodplanes}, we get that $ q_0^3+q_0+1 \leq \vert \Pi\cap B\vert \leq q_0^3+q_0^2+q_0+1$ for all planes $\Pi$ through $L$, containing points of $B$ outside of $L$.

By Lemma \ref{size}, $\vert B \vert \leq q_0^h+q_0^{h-1}+q_0^{h-2}+3q_0^{h-3}$, so there are at most $q_0^{h-3}+2q_0^{h-4}$ planes through $L$ containing points of $B$ outside of $L$. Since $P$ lies on at least $q_0^{h-1}-4q_0^{h-2}+1$ $(q_0+1)$-secants, there are at least two planes $\Pi_1$ and $\Pi_2$ containing at least $q_0^2-6q_0+1$ $(q_0+1)$-secants through $P$.
Suppose that $L'$ is a $(q_0+1)$-secant through $P$, different from $L$, lying on only bad planes. At least one of the planes $\Pi_1,\Pi_2$, say $\Pi_1$, does not contain $L'$.

We will now show that for all $k\leq h-2$, there exists a $k$-space through $\Pi_1$, not containing $L'$, containing at least $q_0^{k}-6q_0^{k-1}$ $(q_0+1)$-secants through $P$. For $k=2$, the statement is true, hence, suppose it holds for all $k\leq j<h-2$. Let $\Pi'$ be a $j$-space through $\Pi_1$, not containing $L'$ and containing at least $q_0^j-6q_0^{j-1}$ $(q_0+1)$-secants through $P$.

Let $\vert \Pi'\cap B\vert=A$, then a $(j+1)$-space $\Pi''$ through $\Pi'$, containing a point of $B$, not in $\Pi'$, contains at least $(q_0-1)A+1$ points of $B$, not in $\Pi'$. We see that the number of $(j+1)$-spaces containing a point of $B$, not in $\Pi'$, is maximal if the number of points in $\Pi'$ is minimal. Since $\vert B\cap \Pi_1\vert \geq q_0^3+q_0+1$, $\vert B \cap \Pi'\vert \geq (q_0^{3}+q_0+1)q_0^{j-3}+1$. This implies that the number of points of $B$ in such a $(j+1)$-space, outside of $\Pi'$ is at least $q_0^{j+1}-q_0^{j}+q_0^{j-1}-q_0^{j-3}+q_0$. Since $\vert B \vert \leq q_0^h+q_0^{h-1}+q_0^{h-2}+3q_0^{h-3}$, the number of such $(j+1)$-spaces is at most $q_0^{h-j-1}+2q_0^{h-j-2}+4q_0^{h-j-3}$. At most $(q_0^{j+1}-1)/(q_0-1)$ $(q_0+1)$-secants through $P$ lie in $\Pi'$. Suppose that all $(j+1)$-spaces through $\Pi'$, except possibly $\langle \Pi',L\rangle$, contain at most $q_0^j-6q_0^{j-1}$ $(q_0+1)$-secants through $P$, not in $\Pi'$, then the number of $(q_0+1)$-secants through $P$ is at most
$$(q_0^{h-j-1}+2q_0^{h-j-2}+4q_0^{h-j-3}-1)(q_0^j-6q_0^{j-1})+(q_0^{j+1}-1)/(q_0-1),$$
a contradiction if $j<h-2$, since there are at least $q_0^{h-1}-4q_0^{h-2}+1$ $(q_0+1)$-secants through $P$.
We may conclude, by induction, that there exists an $(h-2)$-space $\Pi''$, not through $L'$, that contains at least $q_0^{h-2}-6q_0^{h-3}$ $(q_0+1)$-secants through $P$. Since $L'$ does not lie in $\Pi''$, this implies that there are at least $q_0^{h-2}-6q_0^{h-3}$ different planes through $L'$ that each have at least $q_0^3$ points outside of $L$, a contradiction since $\vert B \vert \leq q_0^h+q_0^{h-1}+q_0^{h-2}+3q_0^{h-3}$. This implies that there is at most one line through $P$ that lies on only bad planes.
\end{proof}

\section{The proof of the main theorem}

\begin{lemma}\label{span} Assume $h>3$ and $q_0>5h-11$. Denote the $(q_0+1)$-secants, not lying on only bad planes, through a point $P$ of $B$ that lies on at least one $(q_0+1)$-secant, by $L_1,\ldots,L_{s}$. Let $x$ be a point of the spread element corresponding to $P$ in $\PG(h(n+1)-1,q_0)$ and let $\ell_i$ be the line through $x$ such that $\B(\ell_i)=L_i\cap B$. Let $\L=\{\ell_1,\ldots,\ell_s\}$, then $\langle \L\rangle$ has dimension $h$. 
\end{lemma}

\begin{proof}
From Lemma \ref{lemma1} and Lemma \ref{eenslecht} we get that $s$ is at least $q_0^{h-1}-4q_0^{h-2}+1-1=q_0^{h-1}-4q_0^{h-2}$. From Lemmas \ref{eerste}(ii) and \ref{goodplanes}, we get that through every line $L_i$, $i=1,\ldots,s$, there are at least $q_0^{h-2}-4q_0^{h-3}+1$ good planes, say $\Pi_{ij}$, $j=1,\ldots,t$, such that $B\cap \Pi_{ij}=\B(\pi_{ij})$, for a plane $\pi_{ij}$ through $\ell_i$. Denote the set of planes $\{\pi_{ij}, 1\leq i\leq s, 1\leq j \leq t\}$ by $\mathcal{V}$, and the set of lines $\{\ell_1,\ldots,\ell_s\}$ by $\mathcal{L}$.

A fixed plane $\pi_{ij}$ of $\mathcal{V}$, say $\pi_{11}$, contains $q_0+1$ lines of $\mathcal{L}$, say $\ell_1,\ldots,\ell_{q_0+1}$. The lines $\ell_1,\ldots,\ell_{q_0+1}$ lie on a set of at least $(q_0+1)(q_0^{h-2}-4q_0^{h-3})+1$ different planes of $\mathcal{V}$. On these planes, there lies a set $\mathcal{P}$ of at least $(q_0+1)(q_0^{h-2}-4q_0^{h-3})q_0^2$ different points $y_1,\ldots,y_u$, not in $\pi_{11}$, such that $\B(y_i)\subset B$.

We claim that $\B(y_r)=\B(y_r')$ implies that $y_r=y_r'$ for $y_r$ and $y_r'$ in $\mathcal{P}$ ($\ast$). We know that $y_r$ lies on $\pi_{ij}$ and $y_r'$ lies on $\pi_{i'j'}$ for some $i,i',j,j'$. Since $\B(\pi_{ij})=B\cap \Pi_{ij}$ and $\B(\pi_{i'j'})=B\cap \Pi_{i'j'}$, the lines $\langle \B(xy_r)\rangle$ and $\langle \B(xy_r')\rangle$ are $(q_0+1)$-secants to $B$. Since we assume that $\B(y_r)=\B(y_r')$, these $(q_0+1)$-secants coincide. Moreover, $\B(xy_r)\subset B$ and $\B(xy_r')\subset B$, so $xy_r$ and $xy_r'$ are transversal lines through the same regulus, which forces $y_r=y_r'$. This proves our claim, hence, different points of the point set $\mathcal{P}$ give rise to different points of $B$.

We will prove that, for all $2\leq k\leq h$ there exists a $k$-space through $x$ with at least $q_0^{k-1}-(5k-11)q_0^{k-2}$ lines of $\L$. The existence of $\pi_{11}$ proves this statement for $k=2$. Assume, by induction, that there exists a $j$-space through $x$, say $\nu$, where $j<h-1$, containing at least
 $q_0^{j-1}-(5j-11)q_0^{j-2}$ lines of $\L$.

We will now count the number of couples ($\ell\in \L$ contained in $\nu$, $r$ a point, not in $\nu$ with $\langle r, \ell\rangle\in \mathcal{V}$).
The number of lines of $\mathcal{L}$ in $\nu$ is at least $q_0^{j-1}-(5j-11)q_0^{j-2}$, the number of points $r\notin \nu$ with $\langle r, \ell\rangle\in \mathcal{V}$ for some fixed $\ell$, 
is at least $(q_0^{h-2}-4q_0^{h-3})q_0^2-(q_0^{j+1}-1)/(q_0-1)$. The number of points $r$ with $\langle r, \ell\rangle\in \mathcal{V}$, is by ($\ast$) at most $|B|$, hence, the number of points $r\notin \nu$ with $\langle r, \ell\rangle\in \mathcal{V}$ is at most $|B|-(q_0^{j-1}-(5j-11)q_0^{j-2})q_0-1$.

 Hence, there is a point $r$, lying on (say) $X$ different planes $\langle r,\ell\rangle$ of $\mathcal{V}$ with
$$X\geq \frac{(q_0^{j-1}-(5j-11)q_0^{j-2})(q_0^{h}-4q_0^{h-1}-(q_0^{j+1}-1)/(q_0-1))}{q_0^h+q_0^{h-1}+q_0^{h-2}+3q_0^{h-3}-q_0^{j}+(5j-11)q_0^{j-1}-1}.$$ This last expression is larger than $q_0^{j-1}-(5(j+1)-11)q_0^{j-2}$, if $h>3$, for all $j\leq h-1$.

This implies that the $j+1$-space $\langle r,\nu\rangle$, contains at least $(q_0^{j-1}-(5(j+1)-11)q_0^{j-2})q_0+1$ lines of $\mathcal{L}$, hence, by induction, we find an $h$-dimensional-space through $x$ containing at least $q_0^{h-1}-(5h-11)q_0^{h-2}$ lines of $\mathcal{L}$.

Suppose now that there is a line of $\mathcal{L}$, say $\ell_s$, not in this $h$-space $\xi$. By Lemma \ref{goodplanes}, there are at least $q_0^{h-2}-4q_0^{h-3}$ planes through $\ell_s$, giving rise to $(q_0^{h-2}-4q_0^{h-3})(q_0^2-q_0)$ points $z$, which are not contained in $\xi$, such that $\B(z)\subset B$. By $(\ast)$, and the fact that there are at least  $(q_0^{h-1}-(5h-11)q_0^{h-2})q_0+1$ points $y$ in $\xi$ such that $\B(y)\subset B$, we get that $\vert B \vert \geq q_0^h+q_0^{h-1}+q_0^{h-2}+3q_0^{h-3}$, a contradiction.

This shows that the dimension of $\langle \mathcal{L}\rangle$ is $h$.
 \end{proof}

We now use the following theorem, which is an extension of \cite[Remark 3.3]{TS:97}.
\begin{theorem}{\rm \cite[Corollary 1]{code}}  \label{unique} A blocking set of size smaller than $2q$ in $\PG(n,q)$ is uniquely reducible to a minimal blocking set.
\end{theorem}

\begin{mtheorem}A small minimal blocking set $B$ in $\PG(n,q)$, with exponent $e$, $q=p^t$, $p$ prime, $q_0:=p^e$, $q_0\geq 7$, $t/e=h$, spanning an $(h-1)$-dimensional space is an $\F_{q_0}$-linear set.
\end{mtheorem}
\begin{proof}
As seen in Lemma \ref{span}, there exists an $h$-dimensional space $\xi$ in $\PG((n+1)h-1,q_0)$, such that $\vert \B(\xi)\cap B\vert \geq q_0^{h}-4q_0^{h-1}+1$. Define $\tilde{B}$ to be the union of $\B(\xi)$ and $B$ and recall that $\B(\xi)$ is a small minimal $\F_{q_0}$-linear blocking set in $\PG(n,q)$. Clearly, $\tilde{B}$ is a blocking set, and its size is equal to $|B|+|\B(\xi)|-|B\cap B(\xi)|$. Hence, $|\tilde{B}|$ is at most $(q_0^{h+1}-1)/(q_0-1)+q_0^{h}+q_0^{h-1}+q_0^{h-2}+3q_0^{h-3}-(q_0^h-4q_0^{h-1}+1)<2q_0^h$. Theorem \ref{unique} shows that $B=\B(\xi)$, so we may conclude that $B$ is an $\F_{q_0}$-linear set.
\end{proof}

By the fact that the exponent of a small minimal blocking set in $\PG(n,q)$ is at least one (see Theorem \ref{1modp}), we get the following corollary.
\begin{corollary}
All small minimal blocking sets in $\PG(n,p^t)$, $p$ prime, $p>5t-11$ spanning a $(t-1)$-space, are $\F_p$-linear.
\end{corollary}

{\bf Acknowledgment:} This research was conducted while the second author was
visiting the Department of Computer Science at E\"otv\"os Lor\'and University
in Budapest. The author thanks all members of the Finite Geometry group for their hospitality during her stay.

The first author acknowledges the partial support of the grants OTKA T-49662,
K-81310, T-67867, CNK-77780, ERC, Bolyai and T\'AMOP.

The authors thank the anonymous referees for their valuable comments.

\end{document}